\documentclass[reqno]{amsart}
\usepackage{amsrefs}
\usepackage{amssymb}
\usepackage{graphicx}
\usepackage{array}
\usepackage{longtable}

\newtheorem{thm}{Theorem}[section]
\newtheorem{lem}[thm]{Lemma}
\newtheorem{prop}[thm]{Proposition}
\newtheorem{cor}[thm]{Corollary}
\numberwithin{equation}{section}

\newcommand{\ie}[0]{\mathrm{i}}

\newcommand{\abs}[1]{\left\lvert#1\right\rvert}
\newcommand{\gfextn}[0]{pdf}

\begin{document}

\title{Explicit lower bounds on $\abs{L(1, \chi)}$}

\author[M.~J. Mossinghoff]{Michael J. Mossinghoff}
\address{Center for Communications Research, Princeton, NJ, USA}
\email{m.mossinghoff@idaccr.org}
\author[V.~V. Starichkova]{Valeriia V. Starichkova}
\address{School of Science, The University of New South Wales, Canberra, Australia}
\email{v.starichkova@adfa.edu.au}
\author[T.~S. Trudgian]{Timothy S. Trudgian}
\address{School of Science, The University of New South Wales, Canberra, Australia}
\email{t.trudgian@adfa.edu.au}
\thanks{Supported by Australian Research Council Future Fellowship FT160100094.}

\date\today
\subjclass[2010]{Primary: 11M20; Secondary: 11Y35, 42A05}
\keywords{Dirichlet character, Dirichlet $L$-function, explicit lower bounds.}

\begin{abstract}
Let $\chi$ denote a primitive, non-quadratic Dirichlet character with conductor $q$, and let $L(s, \chi)$ denote its associated Dirichlet $L$-function.
We show that $\abs{L(1, \chi)} \geq 1/(9.12255 \log(q/\pi))$ for sufficiently large $q$, and that $\abs{L(1, \chi)} \geq 1/(9.69030 \log(q/\pi))$ for all $q\geq2$, improving some results of Louboutin.
The improvements stem principally from the construction, via simulated annealing, of some real trigonometric polynomials having particularly favorable properties.
\end{abstract}

\maketitle

\section{Introduction}\label{almira}

Let $\chi$ denote a primitive, non-quadratic Dirichlet character with conductor $q$, and let $L(s,\chi)$ denote its associated $L$-function.
It is well known \cite[Thm.\ 11.4]{MV} that $L(1,\chi)$ for such characters satisfies
\[
\frac{1}{\log q} \ll \abs{L(1,\chi)} \ll \log q,
\]
and explicit inequalities are known for both the upper and lower bounds.
Explicit upper bounds on $\abs{L(1,\chi)}$ are studied for example in \citelist{\cite{Louboutin04}\cite{PE13}\cite{Ramare01}\cite{Ramare04}}; an explicit lower bound was first established by Louboutin in 1992 \cite{Louboutin92}, who proved that
\[
\abs{L(1,\chi)} \geq \frac{1+o(1)}{\sqrt{295}\log(q/\pi)}
\]
as $q\to\infty$.
Louboutin sharpened this estimate in 2015 \cite{Louboutin15}, showing that
\begin{equation}\label{orlando}
\abs{L(1,\chi)} \geq \frac{1+o(1)}{9.27629\log(q/\pi)}.
\end{equation}
It was also established there that a weaker inequality holds in a broader setting:
\begin{equation}\label{arminio}
\abs{L(1,\chi)} \geq \frac{1}{10\log(q/\pi)}
\end{equation}
for all $q\geq2$.
A key component of the proofs of \eqref{orlando} and \eqref{arminio} involved the selection of an auxiliary nonnegative real trigonometric polynomial having certain properties, and the author determined an optimal selection for qualifying polynomials of the form $(1+\alpha\cos\theta+\beta\cos(2\theta))^2$.
In this article, we obtain improved explicit lower bounds on $\abs{L(1,\chi)}$ by allowing trigonometric polynomials of larger degree.
We use simulated annealing to locate these (rare) polynomials.
We prove the following theorem.

\begin{thm}\label{agrippina}
Let $\chi$ denote a primitive, non-quadratic Dirichlet character with conductor $q$.
Then
\begin{equation}\label{giulio}
|L(1, \chi)| \geq \frac{1 + o(1)}{9.122545 \log(q/\pi)}
\end{equation}
as $q\to\infty$.
In addition,
\begin{equation}\label{cesare}
|L(1, \chi)| \geq \frac{1}{9.690292 \log(q/\pi)}
\end{equation}
for all $q\geq2$.
\end{thm}

Applications for lower bounds on $\abs{L(1,\chi)}$ occur in a number of settings, see for example \citelist{\cite{BW19}\cite{ParkKwon}\cite{OSYT}}.

We remark that the situation is quite different for quadratic characters: in that case it is known only that $|L(1, \chi)| \gg_\epsilon q^{-\epsilon}$ for arbitrary $\epsilon>0$ (see \cite[Thm.\ 11.14]{MV}).
We also add that much better upper and lower bounds on $\abs{L(1,\chi)}$ for both the quadratic and non-quadratic cases are known under the assumption of the Generalized Riemann Hypothesis (GRH)\@.
In that case, it is known that if $\chi$ is a primitive Dirichlet character, then
\begin{equation}\label{guistino}
\frac{12e^\gamma(1+o(1))}{\pi^2 \log\log q} \leq \abs{L(1, \chi)} \leq 2e^\gamma(1+o(1))(\log\log q),
\end{equation}
where $\gamma$ denotes the Euler--Mascheroni constant.
Results of this type date to the work of Littlewood \cite{Littlewood28}, who established the lower bound in \eqref{guistino} for non-principal quadratic characters and the upper bound for arbitrary non-principal Dirichlet characters.
More recent work provides fully explicit upper and lower bounds of the form \eqref{guistino}: such a result covering all non-principal primitive Dirichlet characters with conductor $q\geq10^{10}$ was established in \cite{LLS15}, and its range of validity was extended to conductors $q\geq3$ in \cite{LangTru}.

This article is organized in the following way.
Section~\ref{sosarme} summarizes Louboutin's method, and its requirements on the auxiliary trigonometric polynomials selected.
Section~\ref{floridante} describes our strategy for computing favorable trigonometric polynomials, which allow us to establish \eqref{giulio} for characters of large order.
Section~\ref{creta} then reports on a similar strategy for determining certain trigonometric polynomials that produce \eqref{giulio} for characters of small order.
Section~\ref{ariodante} establishes Theorem~\ref{agrippina}.

\section{Louboutin's Method}\label{sosarme}

We summarize Louboutin's method for obtaining explicit lower bounds on $\abs{L(1,\chi)}$ for a primitive, non-quadratic character $\chi$, and we add some additional observations that facilitate its application.
For $\vec{a}=(a_0, \ldots, a_m)\in\mathbb{R}^{m+1}$ with $m\geq1$, let $S(\vec{a}, \theta)$ denote the even trigonometric polynomial
\[
S(\vec{a}, \theta) = a_0 + 2\sum_{k=1}^m a_k \cos(k\theta),
\]
and similarly for a character $\chi$ and integer $n$ let $S(\vec{a}, \chi, n)$ denote
\begin{equation}\label{berenice}
S(\vec{a}, \chi, n) = a_0 + 2\sum_{k=1}^m a_k \Re(\chi^k(n)).
\end{equation}
If $\chi(n)\neq0$, then certainly $S(\vec{a}, \chi, n)=S(\vec{a}, \arg(\chi(n)))$, but if $\chi(n)=0$ one may still have $\chi^k(n)\neq0$: suppose $\chi$ has conductor $q$ and $\chi=\prod_{p\mid q} \chi_p$ with $p$ prime, where each $\chi_p$ has conductor $p^{e_p}$ for some positive $e_p$.
If $\chi_p$ has order $d_p$ then $\chi_p(p)=0$ but $\chi_p^{jd_p}(p)\neq0$ for integers $j$, in fact, $\chi_p^{k}(p)\neq0$ precisely when $k$ is an integral multiple of $d_p$.
Thus, $\chi^k(n)\neq0$ if and only if $d_p\mid k$ for each $p\mid\gcd(q,n)$.

We say a vector $\vec{a} = (a_0, \ldots, a_m)\in\mathbb{R}^{m+1}$ with $m\geq1$ is \textit{admissible} if $a_i\geq0$ for each $i$ and $a_0<2a_1$.
For an admissible $\vec{a}$, we set
\begin{equation}\label{deidamia}
\begin{split}
A &= \left(1-\frac{1}{\sqrt{5}}\right)\sum_{k=2}^m a_k,\\
c &= \frac{a_0 + 2a_1 + 4A + \sqrt{(2a_1 - a_0)^2 + 16(a_1 + A)^2}}{4(2a_1-a_0)},\\
s(q) &= 1+\frac{1}{c\log(q/\pi)},\\
\lambda &= \frac{12 a_1 c^2}{\pi^2e^{1/2c}\left((2a_1-a_0)c-(a_1+A)\right)}.
\end{split}
\end{equation}
We also define the function
\begin{equation}\label{ezio}
F(s) = \frac{\Gamma(s/2)\zeta(2s)}{\Gamma(1/2)\zeta(2)\zeta(s)(s-1)}
\end{equation}
and note that $F(s) = 1 + O(s-1)$.
Finally, we require a few definitions pertaining to constraints on the real number $s$ selected as input to $F$, in the presence of a given character $\chi$.
For this, we first define $\epsilon(\chi)$ to encode whether $\chi$ is an even or an odd character:
\[
\epsilon(\chi) = \begin{cases}
0, & \textrm{if $\chi(-1)=1$},\\
1, & \textrm{if $\chi(-1)=-1$}.
\end{cases}
\]
Next, for an admissible $\vec{a}\in\mathbb{R}^{m+1}$ and primitive Dirichlet character $\chi$, for $s>1$ let
\begin{gather*}
\tau(s) = \frac{1+\sqrt{1+4s^2}}{2},\\
G(\vec{a}, \chi, s) = \sum_{k=2}^m a_k\left(\frac{\Gamma'}{\Gamma}\left(\frac{s+\epsilon(\chi^k)}{2}\right)-\frac{1}{\sqrt{5}}\frac{\Gamma'}{\Gamma}\left(\frac{\tau(s)+\epsilon(\chi^k)}{2}\right)-\frac{2}{\sqrt{5}}\frac{\zeta'}{\zeta}\left(\tau(s)\right)\right).
\end{gather*}
(Note that $G$ is independent of $a_0$ and $a_1$.)
Then, given $\vec{a}$ and $\chi$, select $s_0>1$ so that
\[
G(\vec{a}, \chi, s) \leq 0
\]
for $1<s\leq s_0$.
We may now state Louboutin's main result.

\begin{thm}[Louboutin \cite{Louboutin15}]\label{serse}
Let $\chi$ be a primitive, non-quadratic Dirichlet character with conductor $q$ and order greater than $m$, let $\vec{a}\in\mathbb{R}^{m+1}$ be admissible, and suppose that $S(\vec{a}, \chi, n) \geq0$ for all integers $n$.
Suppose further that $s(q) \leq\min(s_0, 1.92326)$.
Then, using the notation from \eqref{deidamia} and \eqref{ezio},
\[
\abs{L(1,\chi)} \geq \frac{F(s(q))}{\lambda \log(q/\pi)} = \frac{1+o(1)}{\lambda \log(q/\pi)}.
\]
\end{thm}

Louboutin showed in \cite[Lem.\ 10]{Louboutin15} that if $\chi$ is even or if $m=2$, then $s_0=2.97675$ suffices.
We first augment this by showing that a similar selection suffices for odd characters.

\begin{lem}\label{rodrigo}
Let $\chi$ be a primitive, non-quadratic odd Dirichlet character with order $m\geq2$, let $\vec{a}\in\mathbb{R}^{m+1}$ be admissible, and suppose $S(\vec{a},\chi,-1) \geq 0$.
Then $G(\vec{a}, \chi, s)\leq0$ for $1<s\leq 2.28266$.
\end{lem}

\begin{proof}
Since $\chi$ is odd, we have that $\epsilon(\chi^k)$ is $0$ when $k$ is even and $1$ when $k$ is odd.
Thus
\[
G(\vec{a},\chi,s) = G_0(s)\sum_{\substack{2\leq k\leq m\\\textrm{$k$ even}}} a_k + G_1(s)\sum_{\substack{2\leq k\leq m\\\textrm{$k$ odd}}} a_k,
\]
where
\begin{align*}
G_0(s) &= \frac{\Gamma'}{\Gamma}\left(\frac{s}{2}\right)-\frac{1}{\sqrt{5}}\frac{\Gamma'}{\Gamma}\left(\frac{\tau(s)}{2}\right)-\frac{2}{\sqrt{5}}\frac{\zeta'}{\zeta}\left(\tau(s)\right),\\
G_1(s) &= \frac{\Gamma'}{\Gamma}\left(\frac{s+1}{2}\right)-\frac{1}{\sqrt{5}}\frac{\Gamma'}{\Gamma}\left(\frac{\tau(s)+1}{2}\right)-\frac{2}{\sqrt{5}}\frac{\zeta'}{\zeta}\left(\tau(s)\right).
\end{align*}
One may verify that $G_0(s)\leq0$, $G_1(s)\geq0$, and $G_1(s)\leq-G_0(s)$ for $1<s\leq2.28266$.
Thus, the desired inequality $G(\vec{a}, \chi, s)\leq0$ for $s$ in the prescribed range follows from the inequality
\[
\sum_{\substack{2\leq k\leq m\\\textrm{$k$ odd}}} a_k \leq
\sum_{\substack{2\leq k\leq m\\\textrm{$k$ even}}} a_k,
\]
and this is established easily by combining the given facts that $S(\vec{a},\chi,-1) \geq 0$ and $a_0 < 2a_1$.
\end{proof}

We may therefore replace the quantity $\min(s_0, 1.92326)$ in Theorem~\ref{serse} with $1.92326$ for the cases of interest.

Using Theorem~\ref{serse}, Louboutin established \eqref{orlando} in two principal steps.
First, he required a real trigonometric polynomial $S(\vec{a},\theta)$, with $\vec{a}\in\mathbb{R}^{m+1}$ admissible, with the property that $S(\vec{a},\theta)\geq0$ for all real $\theta$.
In fact, an apparently stronger condition was required for $S(\vec{a},\theta)$ in light of the discussion after \eqref{berenice}, since $\chi^k$ might vanish on all but an arithmetic progression of integers $k$.
In \cite{Louboutin15} it was therefore required more generally that each of the polynomials
\begin{equation}\label{admeto}
S_d(\vec{a},\theta) = a_0 + \sum_{\substack{1\leq k\leq m\\d\mid k}} a_k\cos(k\theta)
\end{equation}
satisfy
\begin{equation}\label{hercules}
S_d(\vec{a},\theta) \geq 0
\end{equation}
for all real $\theta$, for $1\leq d\leq m$.
Louboutin established in \cite[Lem.\ 15]{Louboutin15} that if \eqref{hercules} holds for each such $d$, then $S(\vec{a},\chi,n)\geq0$ for all integers $n$ and all primitive non-quadratic Dirichlet characters $\chi$ with order greater than $m$.

We note here however that it is enough to establish the case $d=1$ of \eqref{hercules}, that is, that $S(\vec{a},\theta)\geq0$ for all real $\theta$.
The other cases follow automatically, as shown by the following averaging argument.

\begin{prop}\label{teseo}
Suppose $\vec{a}\in\mathbb{R}^{m+1}$, $D\leq m$ is a positive integer, and that there exist real numbers $\theta_0$, $L$, and $U$ such that $L\leq S(\vec{a}, \theta_0+2 \pi k/D)\leq U$ for $0\leq k<D$.
Then for each $1\leq d \leq D$ with $d\mid D$ we have
\[
L \leq S_d\left(\vec{a}, \theta_0 + \frac{2 \pi k}{d}\right) \leq U
\]
for $0\leq k<d$.
\end{prop}

\begin{proof}
Suppose $d$ is a positive integer with $d\mid D$.
Then $L\leq S(\vec{a}, \theta_0 + 2 \pi \ell/d)\leq U$ for $0\leq \ell<d$, and
\begin{align*}
\frac{1}{d} \sum_{\ell=0}^{d-1} {S\left(\vec{a}, \theta_0 + \frac{2 \pi \ell}{d}\right)}
&= \frac{1}{d} \Re \sum_{\ell=0}^{d-1} \sum_{j=0}^m a_j e^{\ie j(\theta_0+2\pi \ell/d)}\\
&= \frac{1}{d} \Re\sum_{j = 0}^{m} {a_j e^{\ie j \theta_0}} \sum_{\ell=0}^{d-1} {e^{2 \pi j\ell \ie/d}} = S_d(\vec{a}, \theta_0),
\end{align*}
so the result follows after replacing $\theta_0$ by $\theta_0 + 2 \pi k/d$, $0 \leq k < d$, in the formula above.
\end{proof}

Our required result then follows easily.

\begin{cor}\label{ottone}
If $\vec{a}\in\mathbb{R}^{m+1}$ and $S(\vec{a},\theta)\geq0$ for all real $\theta$, then for $1\leq d\leq m$ we have $S_d(\vec{a},\theta)\geq0$ for all real $\theta$.
\end{cor}

\begin{proof}
Select $L=0$ in Proposition~\ref{teseo} (and $U$ sufficiently large) and apply the statement for $\theta_0\in[0,2\pi/d)$.
\end{proof}

In \cite{Louboutin15}, Louboutin chose $m=4$ and considered trigonometric polynomials of the form
\begin{equation}\label{poro}
S(\vec{a},\theta) = \left(1+\alpha\cos(\theta)+\beta\cos(2\theta)\right)^2,
\end{equation}
selecting $\alpha$ and $\beta$ in order to minimize the value of $\lambda$ while maintaining admissibility, as well as the nonnegativity of $S(\vec{a},\theta)$.
The values $\alpha=1.601$ and $\beta=0.709$ were found to produce $c=3.42289\ldots$ and $\lambda=9.27628\ldots$ in \eqref{orlando}.
This established \eqref{orlando} for characters of order greater than $4$.

Second, special polynomials were required to cover the cases of characters with order at most $m=4$.
For example, for order $4$ the selection was $\vec{a}=(2,2,1)$ so that $S(\vec{a},\chi,n)=2+4\Re(\chi(n))+2\Re(\chi^2(n))$, which is nonnegative for characters $\chi$ of order $4$, since $\chi(n)\in\{1,-1,i,-i\}$.
This choice produced $c=2.32703\ldots$ and $\lambda=5.04595\ldots$\,, which is well below the required bound of $9.27628$.

In the closing remarks of \cite{Louboutin15}, Louboutin proposed that one might improve \eqref{orlando} by analyzing another family of nonnegative real trigonometric polynomials besides \eqref{poro}, for example, something with three or more parameters.
Indeed we take such an approach here, but with polynomials of substantially larger degree, producing many more parameters to optimize over.
In the next section we describe the construction of a favorable nonnegative trigonometric polynomial $S(\vec{a},\theta)$ with $\vec{a}\in\mathbb{R}^{m+1}$ admissible, where $m=32$.
Corollary~\ref{ottone} greatly simplifies our constraints: we need only ensure that $S(\vec{a},\theta)\geq0$ on $\mathbb{R}$ for our choice of $\vec{a}$.
Then in Section~\ref{creta} we must cover the characters with order at most $32$ using special polynomials, which require some additional searches.

\section{Characters with large order}\label{floridante}

Given an integer $m\geq2$, we would like to determine an admissible $\vec{a}\in\mathbb{R}^{m+1}$ so that the associated real trigonometric polynomial $S(\vec{a},\theta)\geq0$ for all real $\theta$, and so that the constant $\lambda$ calculated from \eqref{deidamia} is as small as possible.
If $m$ is small, we could employ analytic strategies to determine an optimal polynomial, but we would explore only a small portion of the available space.
With a larger value of $m$, we might hope to find better solutions, but direct analysis may not be so tractable.
We can proceed however by using numerical optimization methods.
We employ the method of simulated annealing here to determine vectors $\vec{a}$ of substantial length that achieve small values for $\lambda$.

In our method, for a fixed degree $m$ we begin by constructing an admissible $\vec{a}\in\mathbb{R}^{m+1}$ for which the polynomial $S(\vec{a},\theta)\geq0$ for all real $\theta$.
To achieve this starting point, we construct our polynomial as the square of the modulus of a randomly constructed trigonometric polynomial.
Set $b_0=1$, and for $1\leq k\leq m$ choose a random value for $b_k$ uniformly from the real interval $[0,B]$, where $B$ is a parameter that may be specified at run time.
Let
\[
g(\theta) = \Big|\sum_{k=0}^m b_k e^{ik\theta}\Big|^2.
\]
We may then compute $a_0$, \ldots, $a_m$ so that $S(\vec{a},\theta)=g(\theta)$: the $a_k$ are simply the (aperiodic) autocorrelations of the $b_k$:
\[
a_k = \sum_{j=0}^{m-k} b_j b_{j+k}.
\]
Each $a_k$ is nonnegative by construction, so we need only ensure that $a_0<2a_1$ for $\vec{a}$ to be admissible.
If this is not the case, we simply restart and select new values for $b_1$, \ldots, $b_m$.
We then proceed with the optimization phase.

In simulated annealing, one aims to optimize an objective function over a given space by an exploratory process: in our application we aim to minimize the value of $\lambda$ from \eqref{orlando} over the subset of $\mathbb{R}^{m+1}$ consisting of admissible points where $S(\vec{a},\theta)\geq0$ on $\mathbb{R}$.
At each step, one moves from one's current point in the space to a nearby point, and evaluates the objective function there.
If this produces a smaller value for $\lambda$, one makes this the new current point and proceeds to the next iteration.
If it produces a larger value, then one moves to the new point only with a particular probability, which depends on the current value of another parameter of the method known as the \textit{temperature}.
If the current temperature is $T$, and the change witnessed in our objective function is $\Delta\lambda>0$, then we keep this step with probability $e^{-(\Delta\lambda)/T}$.
In early stages of the method, the temperature is set to a high value, so that we keep many steps that make the value of our objective function worse.
This promotes exploration of the space, and helps prevent the optimizer from stalling at a local minimum.
As the method proceeds, the temperature is gradually decreased, so that we move more often in ways that improve the value of the objective function.
A final phase in effect sets $T=0$, so that one employs strict greedy descent.

Our method depends on the values of several parameters.
In addition to the coefficient bound $B$ described above, one must specify:
\begin{itemize}
\item $S_1$ and $S_2$, the largest and smallest allowable maximum step sizes,
\item $\rho>0$, for controlling changes to the maximum step size,
\item $\ell$, the number of positive temperature values considered,
\item the annealing schedule $T_1>\cdots>T_\ell>T_{\ell+1}=0$, and
\item $M$, the number of simulated annealing trials per temperature value and per maximum step size.
\end{itemize}
In our method, after selecting qualifying $b_1$, \ldots, $b_m$ and calculating the associated $a_0$, \ldots, $a_m$, we start by setting the maximum step size $S=S_1$ and the temperature $T=T_0$.
In each of $M$ trials we then select an integer $k\in[1,m]$ and a real value $s\in[-S,S]$ uniformly at random, and add $s$ to $b_k$.
We then update each component of $\vec{a}$ to reflect this change: this requires just $O(m)$ work.
If the new value of $\vec{a}$ is not admissible, we undo the change and proceed to the next iteration.
If it is admissible, we know $S(\vec{a},\theta)\geq0$ everywhere by construction, so we compute the value of $\lambda$ at this new point.
We then either keep the new point or undo this step, depending on the sign of $\Delta\lambda$ and the value of the current temperature, according to the rules for simulated annealing.

After $M$ steps at the current temperature $T_i$, we perform another $M$ steps at the next value in the annealing schedule, $T_{i+1}$, and continue this until we complete the batch at $T=T_{\ell+1}=0$.
After this, we divide the value of the maximum step size $S$ by $1+\rho$ and repeat the procedure, starting again with $T=T_1$, provided the new value of $S$ exceeds $S_2$.
When $S$ dips below this value, the process halts.
Throughout, we display the state of the $a_k$ (suitably normalized) and the value of the objective function $\lambda$ whenever we achieve a new record which lies below a given threshold of interest.

We performed more than $10^5$ trials of our procedure at several values of $m$.
The parameters $B$, $S_1$, $S_2$, $\rho$, $M$, and the annealing schedule $T_1$, \ldots, $T_\ell$ varied across these runs, but we often selected $B\in[100,300]$, $S_1\in[2.7,4.5]$, $S_2=10^{-6}$, $\rho\in[.01,.025]$, $\ell\in[8,12]$, $M\in[2000,12000]$, $T_0$ near $0.1$, and $1/T_{i+1}=1/T_i+\delta$ with $\delta$ selected from a small range such as $[1.5,1.75]$.

Table~\ref{alceste} lists the best values found by our method for several values of $m\leq32$.
We employ the value shown here for $m=32$ to establish \eqref{giulio} in Theorem~\ref{agrippina} for characters with order larger than $32$.
The coefficients $\vec{a}$ for this case appear in Figure~\ref{semele}, and the minuscule oscillations of the corresponding real trigonometric polynomial $S(\vec{a},\theta)$ over $[\pi/2,\pi]$ are illustrated in Figure~\ref{galatea}.

\begin{table}[tb]
\caption{Best values for $\lambda$ found at several values of $m$ ($\lambda$ values truncated).}\label{alceste}
\begin{tabular}{|rc|}\hline
$m$ & $\lambda$\\\hline
 $8$ & $9.14800363$\\
$12$ & $9.12993037$\\
$16$ & $9.12475747$\\
$20$ & $9.12328038$\\
$24$ & $9.12292422$\\
$28$ & $9.12268475$\\
$32$ & $9.12254419$\\
\hline
\end{tabular}
\end{table}

\begin{figure}[tb]
\caption{The admissible value $\vec{a}$ for the case $m=32$ of Table~\ref{alceste}, yielding $\lambda=9.1225441\ldots$\,.}\label{semele}
\small
\begin{equation*}
\begin{split}
\vec{a} = \bigl(&1, 0.873189274511716, 0.570782794693574, 0.260015116563811,\\
& 0.0656922507590789, 2.67697182147118\cdot 10^{-6}, 7.44683564477752\cdot 10^{-13},\\
& 0.00696267550009933, 0.00419060975244452, 6.2011676780189\cdot 10^{-7},\\
& 1.40085037941165\cdot 10^{-6}, 0.00125084111895124, 0.000891793444870752,\\
& 2.96130352000425\cdot 10^{-10}, 4.45789856003874\cdot 10^{-6}, 0.000377374247406402,\\
& 0.000288379953409868, 2.328881379567\cdot 10^{-6}, 3.87293028193371\cdot 10^{-6},\\
& 0.000141925431550028, 0.000109824949671327, 8.60368523724788\cdot 10^{-10},\\
& 9.63285107422078\cdot 10^{-7}, 5.45973432309874\cdot 10^{-5}, 4.42734462821885\cdot 10^{-5},\\
& 6.98268956353718\cdot 10^{-6}, 2.09618267164813\cdot 10^{-6}, 1.37818763599687\cdot 10^{-5},\\
& 7.76556929147892\cdot 10^{-6}, 8.50464668334544\cdot 10^{-10}, 1.27987067045419\cdot 10^{-8},\\
& 1.50505483825263\cdot 10^{-6}, 6.83922632017741\cdot 10^{-8}\bigr).
\end{split}
\end{equation*}
\end{figure}

\begin{figure}[tb]
\caption{$S(\vec{a},\theta)$ over $[\pi/2,\pi]$ for $\vec{a}$ in Figure~\ref{semele}.}\label{galatea}
\begin{center}
\includegraphics[width=3.25in]{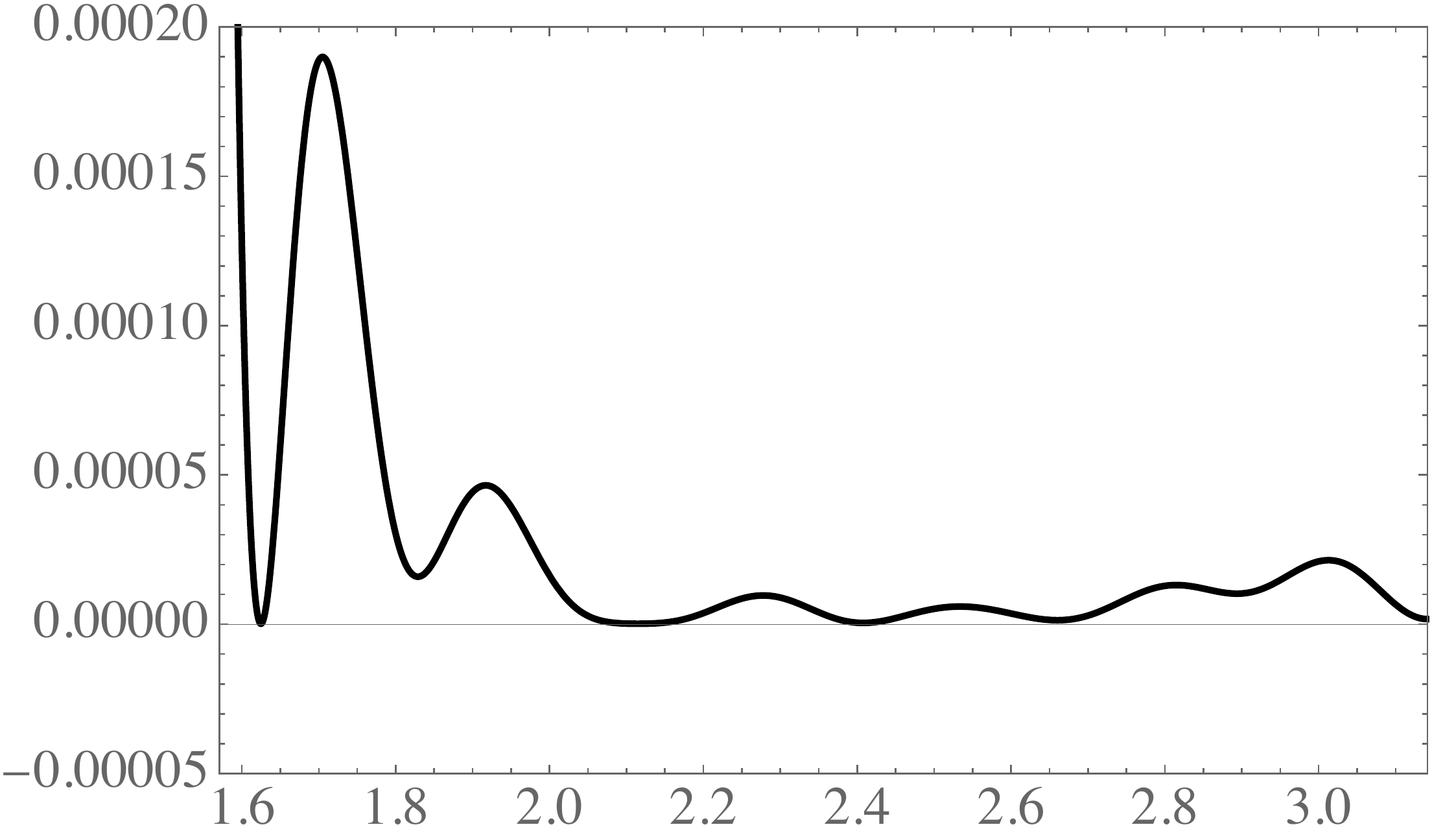}
\end{center}
\end{figure}

We remark that a similar simulated annealing procedure was employed in \cite{MT15} to determine a nonnegative real trigonometric polynomial to assist with the determination of an improved explicit zero-free region for the Riemann zeta function of the form  $\sigma>1-1/(R_0\log t)$, where $\sigma+it$ denotes a point in the complex plane.
While a different objective function was employed there, the polynomial of degree $16$ determined in that work qualifies in this problem as well, and in fact produces a very good value of $9.12725$ in \eqref{giulio}.

\section{Characters with small order}\label{creta}

We now turn to establishing \eqref{giulio} for primitive, non-quadratic Dirichlet characters of order at most $32$.
We note that it follows from the proof of \cite[Lem.\ 15]{Louboutin15} that if \eqref{hercules} holds for all $\theta = 2 \pi k / D$, $0 \leq k < D$ and for $1\leq d\leq m$, $d \mid D$, then $S(\vec{a},\chi,n)\geq0$ for all integers $n$ and all primitive non-quadratic Dirichlet characters $\chi$ with order $D$.

Thus, in view of Proposition~\ref{teseo} and Theorem~\ref{serse}, for each $d\leq32$ we require an admissible $\vec{a}\in\mathbb{R}^{m+1}$ with $m<d$ such that 
\begin{equation}\label{lotario}
S(\vec{a}, 2\pi k/d) \geq 0
\end{equation}
for all integers $k$, and where the associated value of $\lambda$ is at most $9.122685$.
Note we do not require that $S(\vec{a}, \theta) \geq 0$ for all real $\theta$: this additional freedom allows us to find qualifying vectors with smaller $\lambda$ values.
Also, by Proposition~\ref{teseo} (with $\theta_0=0$), we do not require a separate check that the values of $S_{d'}(\vec{a},2\pi k/d)$ from \eqref{admeto} with $d'>1$ are also nonnegative.

Louboutin exhibited qualifying polynomials for orders $d=3$, $4$, $5$, and $6$, and some initial searches using integer values of the coefficients $a_k$ allowed us to find simple qualifying examples for several additional values of $d$.
For smaller values of $d$, simple exhaustive searches with bounded coefficients sufficed.
For certain larger values, we found some success with an iterative procedure: first, find a moderately good value for $\lambda$ with a simple exhaustive search with bounded coefficients, then multiply each coefficient by a small integer, and search in the neighborhood of that point for improved $\lambda$ values.
These methods sufficed for $d\leq10$ and for $12\leq d\leq 16$, and our results are listed in Table~\ref{alessandro}.
We include Louboutin's examples for $3\leq d\leq6$ here as well for the convenience of the reader.

We required a different method for the remaining $d\leq32$.
For these, we adapted our simulated annealing procedure to determine qualifying polynomials.
In this case, we no longer require nonnegativity everywhere, so we remove the sequence $b_k$ from the procedure, and instead we operate on the $a_k$ values directly.
However, in order to apply our simulated annealing procedure, we must begin with a valid configuration, where $S(\vec{a}, 2\pi k/d) \geq 0$ for each $k$, and random selection of the $a_k$ is very unlikely to produce a solution with this property, especially for larger values of $d$.

We can resolve this by applying a preliminary round of simulated annealing, which behaves much like the procedure described in Section~\ref{floridante}, but with another objective function.
For a fixed integer $d$ and a selected value for $m<d$, we aim to minimize
\begin{equation}\label{atalanta}
-\sum_{k=0}^{d-1} \min\left\{0, S\left(\vec{a}, \frac{2\pi k}{d}\right)\right\}
\end{equation}
over admissible $\vec{a}\in\mathbb{R}^{m+1}$.
For each $d\leq32$, we set $m=d-1$, $a_0=1$, and select $a_1\in(1/2,1]$ and $a_i\in[0,1]$ for $2\leq i\leq m$ uniformly at random.
We then apply our simulated annealing procedure with the aim of minimizing the objective function \eqref{atalanta}, and halt as soon as we determine a vector for which this evaluates to $0$.
This determined a vector $\vec{a}$ with the property \eqref{lotario} very quickly for each required $d$.
We then apply the simulated annealing procedure from Section~\ref{floridante}, amended so that we adjust the values $a_i$ directly rather than the $b_i$, and so that we require that \eqref{lotario} remain true for each $k$ in every step.

Our results for each remaining order $d\leq32$ are summarized in Table~\ref{susanna}.
The real trigonometric polynomial $S(\vec{a},\theta)$ selected for the case $d=32$ is displayed in Figure~\ref{radamisto}, along with its values at $\pi k/16$ for integer $k$.

\begin{table}[tb]
\caption{Qualifying trigonometric polynomials for several orders $d\leq16$ having $\lambda<9.12$, determined using heuristic integer searches. (Values of $\lambda$ truncated.)}\label{alessandro}
\begin{tabular}{|c|l|c|}\hline
$d$ & $\vec{a}$ & $\lambda$\\\hline
$3$, $6$ & $(1, 1)$ & $3.72935$\\
$4$ & $(2, 2, 1)$ & $5.04595$\\
$5$ & $(1, 1, 1)$ & $6.38742$\\
$7$ & $(1, 1, 1, 1)$ & $9.06189$\\
$8$ & $(37, 31,18, 5)$ & $9.09363$\\
$9$ & $(22, 19, 12, 5)$ & $9.05800$\\
$10$ & $(26, 23, 16, 8, 2)$ & $9.10161$\\
$12$ & $(196, 169, 106, 44, 9)$ & $9.07653$\\
$13$ & $(93, 81, 53, 24, 6)$ & $9.07653$\\
$14$ & $(136, 120, 82, 40, 11, 0, 0, 1)$ & $9.11790$\\
$15$ & $(476, 426, 301, 157, 48, 0, 0, 13)$ & $9.11703$\\
$16$ & $(375, 324, 205, 87, 19, 0, 1, 1)$ & $9.11712$\\
\hline
\end{tabular}
\end{table}

\begin{figure}[tb]
\caption{The degree $31$ trigonometric polynomial $S(\vec{a},\theta)$ over $0\leq\theta\leq2\pi$ with $\vec{a}$ from Table~\ref{susanna} for the case $d=32$, and its (positive) values at the points $\theta=\pi k/16$ for $0\leq k\leq32$.}\label{radamisto}
\begin{center}
\includegraphics[width=3.25in]{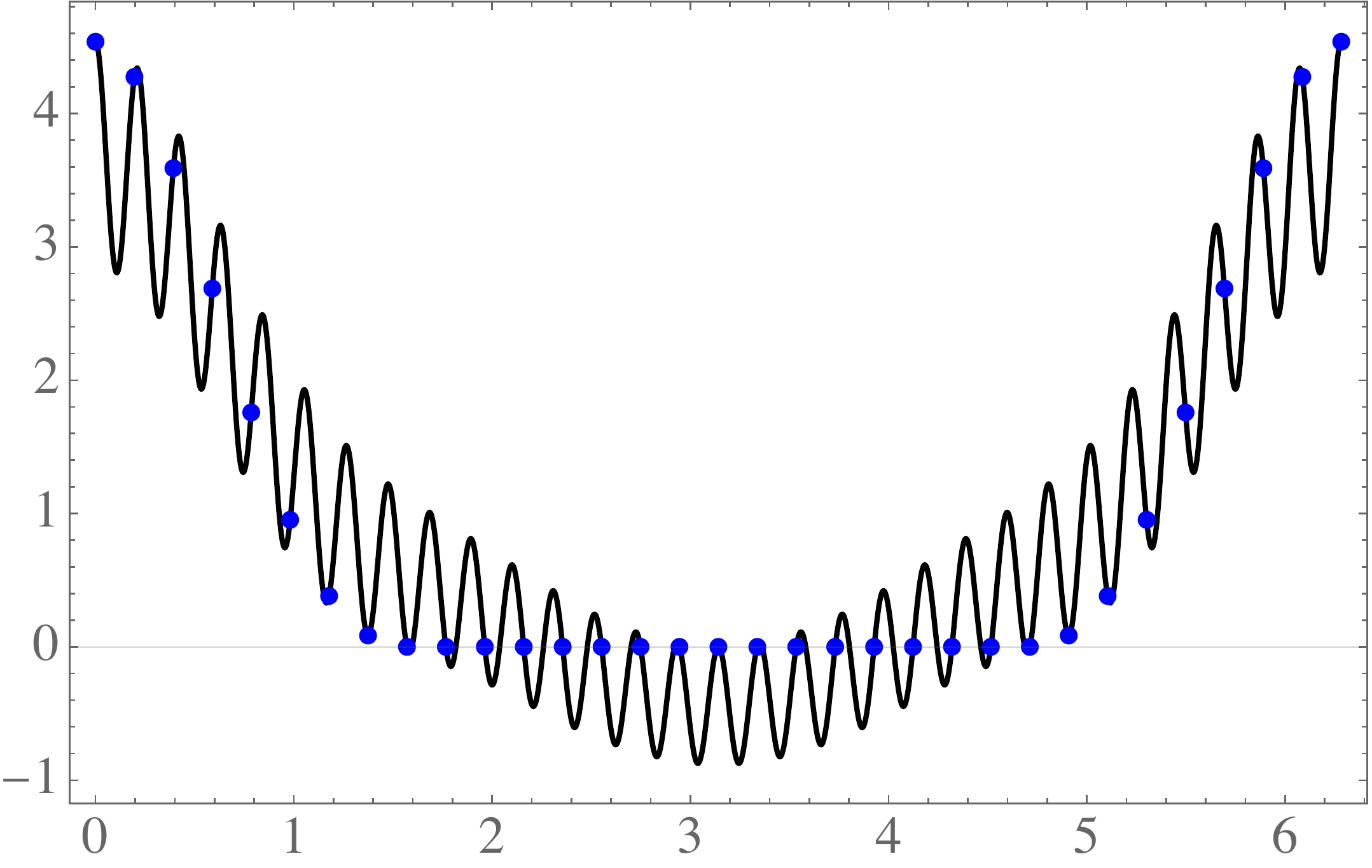}
\end{center}
\end{figure}

\section{Proof of Theorem~\ref{agrippina}}\label{ariodante}

We may now establish Theorem~\ref{agrippina}.

\begin{proof}[Proof of Theorem~\ref{agrippina}]
Our computations in Section~\ref{floridante} and~\ref{creta}, together with Theorem~\ref{serse} and Corollary~\ref{ottone}, establish \eqref{giulio}.
For \eqref{cesare}, Louboutin \cite[Lem.\ 3]{Louboutin15} noted that $F(s)$ in \eqref{ezio} is decreasing for $1<s\leq1.92326$.
Since $s(q)$ from \eqref{deidamia} is decreasing as well, if $q\geq q_0$ then $F(s(q))\geq F(s(q_0))$, so it suffices to compute $\abs{L(1,\chi)}$ for all primitive, non-quadratic Dirichlet characters $\chi$ up to a particular bound $q_0$ in order to obtain a result of the form \eqref{cesare}.
For \cite{Louboutin15}, D.~Platt performed this computation with $q_0=90000$; here Platt kindly extended this computation up to $q_0=10^6$
and determined that
\begin{equation}\label{rinaldo}
0.41292867 \leq \abs{L(1,\chi)}\log(q/\pi) \leq 1.91211802
\end{equation}
over this range, so that $\lambda=2.4218$ suffices for these characters.
\end{proof}

We remark that equality occurs on both sides of \eqref{rinaldo} when $\chi$ is one of the two primitive, non-quadratic characters with conductor $5$.
Platt's calculation employed interval arithmetic and required approximately $400$ core-hours, distributed over a number of $2.6$ GHz Intel Sandy Bridge processors.
We also remark that if one were to extend Platt's computation to $q_0=10^7$, and all $L(1,\chi)$ values remained in the appropriate range, then it would follow immediately that one could replace \eqref{cesare} with
\[
|L(1, \chi)| \geq \frac{1}{9.602277 \log(q/\pi)}
\]
for all $q\geq2$.

Finally, we add that we believe that the current method (without the assumption of GRH) does not allow for much improvement to \eqref{giulio} and \eqref{cesare}.
In particular, we think it is unlikely that one can obtain a value smaller than $9$ in \eqref{giulio} with the method described here.

\begin{longtable}[c]{|c|m{3.8in}|c|}
\caption{Qualifying trigonometric polynomials with degree $m=d-1$ for remaining orders $d\leq32$ having $\lambda<9.1224$, determined using simulated annealing. (Values of $\lambda$ truncated.)}\label{susanna}\\\hline
$d$ & $\vec{a}$ & $\lambda$\\\hline
$11$ & \raggedright\tiny($1$, $0.8925512693$, $0.1707136084$, $0.004260864165$, $0.1078912616$, $3.250649999\cdot10^{-6}$, $0.00001763400605$, $0.0009810003897$, $0.3264771169$, $0.4566384223$, $1.520660622\cdot10^{-6}$) & $9.10488$\\\hline
$17$ & \raggedright\tiny($1$, $0.8651903761$, $0.3013064921$, $0.1032699308$, $0.02640447437$, $8.921912248\cdot10^{-9}$, $0.01308300847$, $0.01303574003$, $7.517344907\cdot10^{-8}$, $5.678907221\cdot10^{-9}$, $0.001589918391$, $0.0009790550999$, $2.214327754\cdot10^{-8}$, $0.01602498072$, $0.1228661349$, $0.2446642912$, $1.414119993\cdot10^{-7}$) & $9.11238$\\\hline
$18$ & \raggedright\tiny($1$, $0.8795259248$, $0.2201432067$, $0.1463916545$, $0.04097288363$, $3.832829719\cdot10^{-8}$, $0.00008084343246$, $0.004829428594$, $0.004274473768$, $0.00001951196368$, $0.0008300131365$, $0.004533961979$, $1.514785772\cdot10^{-8}$, $8.130543691\cdot10^{-7}$, $0.03273466286$, $0.1311416916$, $0.3674170916$, $9.618625927\cdot10^{-7}$) & $9.08309$\\\hline
$19$ & \raggedright\tiny($1$, $0.8879449717$, $0.1811994811$, $0.2255380343$, $0.08398853634$, $5.781192245\cdot10^{-6}$, $4.658968571\cdot10^{-8}$, $0.01322595491$, $0.009787187061$, $5.449748813\cdot10^{-7}$, $1.203874072\cdot10^{-7}$, $0.003515548148$, $0.003373991071$, $5.408926667\cdot10^{-6}$, $1.008688798\cdot10^{-7}$, $0.002316767751$, $0.07726235198$, $0.4293208701$, $3.639386072\cdot10^{-7}$) & $9.10897$\\\hline
$20$ & \raggedright\tiny($1$, $0.8684793735$, $0.5520798472$, $0.0004657800957$, $0.002513205149$, $2.178244004\cdot10^{-6}$, $6.467026508\cdot10^{-7}$, $0.001591507432$, $2.770945051\cdot10^{-6}$, $2.959250305\cdot10^{-6}$, $0.0006662538523$, $1.714696968\cdot10^{-6}$, $0.0001138663591$, $0.001360366391$, $2.756159535\cdot10^{-8}$, $1.439010271\cdot10^{-7}$, $0.05588558022$, $0.2451096527$, $0.005757803016$, $6.296479921\cdot10^{-6}$) & $9.09358$\\\hline
$21$ & \raggedright\tiny($1$, $0.8736906308$, $0.2801126303$, $0.1522622045$, $0.03892071469$, $0.00003406146546$, $2.093505814\cdot10^{-7}$, $0.0005481592265$, $0.0005031715343$, $6.728004616\cdot10^{-7}$, $0.0002362233062$, $0.003413797906$, $2.743464351\cdot10^{-7}$, $0.0005011887291$, $0.004236607157$, $5.739640493\cdot10^{-6}$, $1.799141899\cdot10^{-7}$, $0.02608985671$, $0.1077289740$, $0.2915210411$, $6.917884324\cdot10^{-7}$) & $9.08552$\\\hline
$22$ & \raggedright\tiny($1$, $0.8799001501$, $0.3789556212$, $0.1444756650$, $0.04275889762$, $6.496417685\cdot10^{-7}$, $1.435096926\cdot10^{-6}$, $0.001902723418$, $0.001901166621$, $4.143534197\cdot10^{-6}$, $0.0002236874584$, $0.001042468369$, $3.442971584\cdot10^{-7}$, $0.00001192954997$, $0.004473973979$, $0.007902036128$, $3.382460700\cdot10^{-8}$, $2.897065037\cdot10^{-6}$, $0.03133401592$, $0.1339405520$, $0.2095351689$, $6.123047061\cdot10^{-8}$) & $9.09131$\\\hline
$23$ & \raggedright\tiny($1$, $0.8828844022$, $0.003798035369$, $0.09783570676$, $0.07524040736$, $5.152723017\cdot10^{-6}$, $1.164498815\cdot10^{-7}$, $0.006317412022$, $0.004683368792$, $8.804840105\cdot10^{-6}$, $8.571154308\cdot10^{-7}$, $0.001560150728$, $0.004065023111$, $1.561178361\cdot10^{-6}$, $4.545204678\cdot10^{-7}$, $0.004227138328$, $0.005993063925$, $3.425232822\cdot10^{-6}$, $7.069603870\cdot10^{-6}$, $0.003475731675$, $0.1897459082$, $0.5929177880$, $1.829964404\cdot10^{-6}$)
 & $9.11866$\\\hline
$24$ & \raggedright\tiny($1$, $0.8688729832$, $0.2956776096$, $0.1191086429$, $0.02804353535$, $1.220363244\cdot10^{-6}$, $8.725012271\cdot10^{-7}$, $0.002106193329$, $0.0006740124793$, $1.022927305\cdot10^{-6}$, $0.001131707796$, $0.0006950315492$, $4.299576093\cdot10^{-6}$, $0.0002016536663$, $0.0001620741594$, $1.216831008\cdot10^{-6}$, $0.0003524091083$, $0.001692912240$, $2.397679763\cdot10^{-6}$, $3.595072236\cdot10^{-7}$, $0.03139803864$, $0.1281788801$, $0.2634229981$, $6.372375880\cdot10^{-8}$)
 & $9.11108$\\\hline
 25 &
\raggedright\tiny($1$, $0.8748673118$, $0.1950196202$, $0.1059895960$, $0.03770393885$, $0.00001228276052$, $2.620584954\cdot10^{-6}$, $1.609173157\cdot10^{-6}$, $0.002850847785$, $1.757612884\cdot10^{-7}$, $0.00009453770686$, $0.0002759479928$, $5.231954909\cdot10^{-6}$, $0.00002821718967$, $0.0005583704073$, $0.00005344951615$, $1.218892906\cdot10^{-6}$, $0.001270935523$, $0.007153043561$, $1.485366673\cdot10^{-6}$, $1.572130607\cdot10^{-6}$, $0.02972238974$, $0.1582197112$, $0.3800666426$, $2.381539024\cdot10^{-6}$)
& 9.09925\\\hline
26 &
\raggedright\tiny($1$, $0.8742243889$, $0.2600128838$, $0.1293816056$, $0.01186410167$, $5.793915192\cdot10^{-6}$, $0.00001542514879$, $0.001658778705$, $0.002083552277$, $6.861147037\cdot10^{-8}$, $1.101750175\cdot10^{-6}$, $0.0004016998096$, $0.0005534858583$, $0.00004488693517$, $0.0003609497698$, $0.001079782646$, $0.00007335340012$, $9.088372024\cdot10^{-7}$, $0.002440634586$, $0.005742079883$, $2.165900802\cdot10^{-7}$, $0.00008624011096$, $0.05534320217$, $0.1336952018$, $0.3135980420$, $2.638649768\cdot10^{-7}$)
& 9.11877\\\hline
27 &
\raggedright\tiny($1$, $0.8807947909$, $0.3980844819$, $0.009777754833$, $0.06905801978$, $1.388915652\cdot10^{-6}$, $4.638261249\cdot10^{-7}$, $0.002510605981$, $0.002595168651$, $7.985582689\cdot10^{-8}$, $1.384925696\cdot10^{-6}$, $0.0004638397608$, $0.0008108685459$, $3.566052140\cdot10^{-7}$, $1.124444604\cdot10^{-6}$, $0.002656221184$, $0.003377825224$, $1.979613100\cdot10^{-6}$, $3.053250399\cdot10^{-7}$, $0.004831115452$, $0.008243806784$, $3.790260994\cdot10^{-8}$, $5.341945068\cdot10^{-6}$, $0.006487566572$, $0.2714277162$, $0.1928772315$, $1.274261542\cdot10^{-6}$)
& 9.12126\\\hline
28 &
\raggedright\tiny($1$, $0.8714855947$, $0.3663179884$, $0.07903215268$, $0.04641565762$, $8.231877496\cdot10^{-7}$, $4.873108215\cdot10^{-7}$, $0.005215163639$, $0.002827605566$, $2.121532764\cdot10^{-6}$, $0.0001787092675$, $0.0001071896439$, $0.0001986564610$, $4.205599683\cdot10^{-6}$, $0.00007967410139$, $1.204255330\cdot10^{-6}$, $6.791827042\cdot10^{-7}$, $0.0005650107449$, $2.660578341\cdot10^{-7}$, $8.046097178\cdot10^{-7}$, $0.0003555381829$, $0.0007432214142$, $0.00008107522875$, $2.570035041\cdot10^{-7}$, $0.01678788683$, $0.1759678728$, $0.1998822270$, $4.430517925\cdot10^{-7}$)
& 9.11357\\\hline
29 &
\raggedright\tiny($1$, $0.8737344202$, $0.3728121341$, $0.1388354606$, $0.01977132935$, $1.85446795\cdot10^{-6}$, $1.339281236\cdot10^{-7}$, $0.005097099705$, $0.002292951708$, $2.489493314\cdot10^{-7}$, $3.102568384\cdot10^{-6}$, $0.0007826150139$, $2.335663297\cdot10^{-6}$, $1.128589904\cdot10^{-6}$, $0.0002813042356$, $0.0006420074362$, $5.357218184\cdot10^{-7}$, $7.049865688\cdot10^{-5}$, $1.708384116\cdot10^{-6}$, $6.563682961\cdot10^{-7}$, $6.135278843\cdot10^{-7}$, $0.001607388119$, $0.001747965525$, $1.97981416\cdot10^{-7}$, $6.529877509\cdot10^{-7}$, $0.04639906817$, $0.122526693$, $0.1994805616$, $2.498427309\cdot10^{-7}$) &
$9.11014$\\\hline
30 &
\raggedright\tiny($1$, $0.8741757548$, $0.2190886383$, $0.1205608538$, $0.04432870594$, $1.065576781\cdot10^{-6}$, $3.173651405\cdot10^{-5}$, $0.001723185396$, $0.0002859065925$, $4.592199673\cdot10^{-6}$, $1.38659381\cdot10^{-6}$, $0.0003921001592$, $0.000242754326$, $3.407881947\cdot10^{-6}$, $6.557545669\cdot10^{-8}$, $0.0001259226757$, $1.798181138\cdot10^{-6}$, $1.986131458\cdot10^{-6}$, $0.0008340080575$, $0.001140693639$, $1.888491525\cdot10^{-7}$, $1.367160563\cdot10^{-6}$, $0.004364622043$, $0.00574934825$, $8.892347559\cdot10^{-7}$, $5.828895014\cdot10^{-7}$, $0.02270585484$, $0.1424700207$, $0.3544943584$, $1.355674121\cdot10^{-6}$) &
$9.12091$\\\hline
31 &
\raggedright\tiny($1$, $0.8708961238$, $0.2781783953$, $0.1351050822$, $0.03440019294$, $2.584243905\cdot10^{-6}$, $9.157456098\cdot10^{-7}$, $0.005852967769$, $0.0004171112266$, $5.56705523\cdot10^{-7}$, $3.869823028\cdot10^{-7}$, $0.0006800244848$, $0.000449875875$, $1.016149849\cdot10^{-6}$, $2.963865467\cdot10^{-7}$, $1.046285176\cdot10^{-6}$, $0.0001502505828$, $1.579708679\cdot10^{-6}$, $3.187102505\cdot10^{-6}$, $1.193308902\cdot10^{-5}$, $6.152423786\cdot10^{-5}$, $2.08259328\cdot10^{-6}$, $8.583615973\cdot10^{-7}$, $0.002963017473$, $0.0001017669839$, $2.213388921\cdot10^{-6}$, $2.347374191\cdot10^{-8}$, $0.02840063515$, $0.118647412$, $0.2865887524$, $7.504306267\cdot10^{-7}$) &
$9.12236$\\\hline
32 &
\raggedright\tiny($1$, $0.8717070569$, $0.2711900771$, $0.159573736$, $0.04166237321$, $8.537262209\cdot10^{-7}$, $6.050919503\cdot10^{-6}$, $0.001790952939$, $0.0005462328795$, $2.919551891\cdot10^{-7}$, $8.3294301\cdot10^{-7}$, $0.0005436228576$, $0.0001549815004$, $7.112764795\cdot10^{-7}$, $5.793902734\cdot10^{-5}$, $3.59725415\cdot10^{-7}$, $1.272517949\cdot10^{-6}$, $9.088342903\cdot10^{-5}$, $0.0001105150739$, $5.634257722\cdot10^{-7}$, $1.110676129\cdot10^{-5}$, $3.089507399\cdot10^{-6}$, $1.278155161\cdot10^{-5}$, $8.457119998\cdot10^{-7}$, $0.002864409168$, $0.004407135231$, $7.840012911\cdot10^{-5}$, $1.304461864\cdot10^{-7}$, $0.02209071361$, $0.09644403928$, $0.2957803607$, $1.50731231\cdot10^{-6}$) &
$9.11715$\\\hline
\end{longtable}

\section*{Acknowledgements}

We thank David Platt for his computations of $L(1,\chi)$ for primitive, non-quadratic Dirichlet characters with conductor at most $10^6$.
We also thank Michaela Cully-Hugill and Forrest Francis for helpful remarks, and for checking some of the calculations.

\end{document}